\documentclass[11pt]{article}
\usepackage{amsfonts}

\usepackage[dvips]{graphicx}
\usepackage{latexsym}
\usepackage{amsmath}
\usepackage{amssymb,color}
\usepackage{amsthm}
\usepackage{url}
\usepackage{paralist}
\usepackage{graphicx}

\def\BBox{\kern  -0.2cm\hbox{\vrule width 0.2cm height 0.2cm}}

\usepackage[normalem]{ulem}
\usepackage{setspace} 
\usepackage{hyperref}
\usepackage{cancel}
\usepackage{enumerate}

\newtheorem{theorem}{Theorem}

\newtheorem{lemma}[theorem]{Lemma}
\newtheorem{prop}[theorem]{Proposition}
\newtheorem{defi}[theorem]{Definition}

\newtheorem{obs}[theorem]{Observation}

\newcommand{\Rd}{\mathbb{R}^d}
\newcommand{\bc}{\begin{center}}
\newcommand{\ec}{\end{center}}
\newcommand{\ba}{\begin{align*}}
\newcommand{\ea}{\end{align*}}
\newcommand{\be}{\begin{enumerate}}
\newcommand{\ee}{\end{enumerate}}
\newcommand{\bi}{\begin{itemize}}
\newcommand{\ei}{\end{itemize}}

\def\fl{\mathcal{F}}

\newenvironment{dedication}
   \itshape             

\title{A note on the intersection property for flat boxes and boxicity  in $\mathbb{R}^d$}

\author{Hector Ba\~nos\thanks{hdbanoscervantes@alaska.edu}\\ 
{\small University of Alaska, Fairbanks} \\
Deborah Oliveros\thanks{dolivero@matem.unam.mx}\\
{\small  Instituto de Matem\'{a}ticas}\\
{\small  Universidad Nacional Aut\'{o}noma de M\'{e}xico, M\'{e}xico}
\\}

\begin{document}
\maketitle

\begin{dedication}
\centerline{Celebrating the 70th birthday of T. Bisztriczky, G. Fejes Toth, and A. Makai} 
\end{dedication}

\begin{abstract}
By extending the definition of boxicity,  we extend a Helly-type result given by Danzer and Gr\"unbaum on 2-piercings of families of boxes in $d$-dimensional Euclidean space by lowering the dimension of the boxes in the ambient space. 
\end{abstract}

\section{Introduction}
A $p$-box in $\Rd$ is a rectangular $p$ dimensional parallelotope whose edges are parallel to the coordinate axes in $\Rd$, where $d\geq p$. 
A family of boxes is called $n$-pierceable if there exists a set of $n$ points such that each box contains at least one 
of these points. 
 
For positive integers $d$ and $n$, what is the smallest number $h = h(d, n)$ such that the following property holds: 

\noindent``Every finite family ${\fl }$ of $d$--boxes in $\Rd$ is $n$-pierceable if and only if, 
every  subfamily of cardinality $h$ is $n$-pierceable.'' 

This problem was originally studied by Danzer and Gr\"{u}nbaum in \cite{DanzGrum}, showing in particular
that the following theorem holds for piercing number $2$:

\begin{theorem}\label{DanzGrum} \cite{DanzGrum}For  $d$ odd, $h(d, 2) = 3d$  and for $d$ even $h(d, 2) = 3d -1$. 
\end{theorem}

Results of the type \textquotedblleft if every subset of cardinality $\mu $
of  ${\fl }$  is $n$--pierceable, then the entire family 
${\fl }$ is $n$--pierceable%
\textquotedblright\ are called Helly--Gallai type theorems. Clearly Helly--Gallai type theorems 
are a natural generalization of Helly's classical theorem for $n=1$  when ${\fl }$ 
is a family of convex sets in $\Rd$ and $\mu =d+1$. Results of this type have been widely studied in 
different settings (see, for instance, surveys such as \cite{DanzGrum}, \cite{Echoff93}). However, in the same paper
\cite{DanzGrum} of Danzer and Gr\"{u}nbaum they show that such theorems do not
exist in general, even for the case of families of $d$-boxes in $\Rd$.

In \cite{woe} the authors give an alternative proof of Theorem \ref{DanzGrum} using intersection graphs of families of boxes, 
the fact that for the case $n=1$, $h(d,1) = 2$ in any dimension, analyzing the structure of the complement of odd cycles and
the chromatic number $\chi(G)$ of a graph $G$; see Proposition \ref{2hp} in this paper.

In 1969 F.S. Roberts \cite{ROB} extended the definition of  interval graph  to higher dimensions by considering the intersection graph 
of a family of $d$--boxes in $d$--dimensional Euclidean space $\Rd$ by defining the boxicity of a graph G, denoted $Box (G)$, as the smallest 
positive integer $d$ for which a graph $G$ is the intersection graph of a family of $d$--boxes in $\Rd$. For further reference in this topic and part of 
it's state of art, see  \cite{sunil1}, \cite{sunil2},  \cite{esperet}, and \cite{sunil3}. 

In his work, Roberts gave a characterization of the boxicity of a noncomplete graph $G$ in terms of interval graphs  by showing the following theorem:

\begin{theorem}\label{Roberts}\cite{ROB}
The boxicity of a noncomplete graph $G$ is the minimum positive integer $k$ such that there exists interval graphs $F_1,F_2,\dots ,F_k$ such that $G = F_1\cap F_2\cap \cdots \cap F_k$.
\end{theorem}

In this paper, we introduce the definition of \emph{$p$-boxicity}, $Box_p(G)$, as the minimum dimension $d$
such that a graph $G$ is realizable as the intersection graph of $p$-boxes in $\Rd$ (boxes of dimension $p$ in $\Rd$), and give a generalization of Roberts' result (Theorem \ref{Roberts})  in terms of $p$-boxes. We believe this definition is interesting in its own right and may yield to further research activity. Furthermore, we extend Danzer's and Gr\"{u}nbaum's  theorem for piercing  number $2$ to the family of flat (that is, not necessarily full dimensional) boxes using this new definition. In particular, we prove the following theorem:

\begin{theorem}\label{flatboxes}
 Given any family $\fl$ of $m$-boxes in $\Rd$, $m\leq d$, $\fl$ is $2$-pierceable if and only if every  
subfamily of cardinality $h=h(d,m,2)$ is $2$-pierceable provided

\noindent $h(d,m,2)=\left\{ \begin{array}{rcl}
5& \mbox{for}\noindent&m=1\\
7& \mbox{for}\noindent&m=2\\
3m& \mbox{for}\noindent&m\neq 1$  $odd\\
 3m-1 & \mbox{for}&m\neq 2$ $even\end{array} \right.$
\end{theorem}

\section{Boxicity and $p$-boxicity}
Given a finite family $\fl =\{ X_1,X_2,\dots, X_k\}$ of $p$-boxes (boxes of dimension $p$) in $\Rd$, a graph $G=(V(G),E(G))$ is the intersection graph of $\fl$ if  $V(G):= \{ x_1,x_2,\dots,x_k\}$ and $(x_i,x_j)\in E(G)$ if and only if $X_i\cap X_j\not= \emptyset$.

Recall that given a graph $G$, its {\textit p-boxicity} $Box_p(G)$ is the minimum dimension $d$ such that $G$ is the intersection graph of family of  $p$-boxes in $\mathbb{R}^d$. Observe that $p\leq d$ and  if $p=d$ then $Box_p(G)=Box(G)$. Note that for any graph $G$ and $p\in \mathbb{N}$, if  $p\geq Box(G)$ then $Box_{p}(G)=p$. In general $Box(G)\leq Box_{p}(G)$. We will say that a family $\fl$ of $p$-boxes is a {\it realization} 
of $G$ if $G$ is the intersection graph of $\fl$.

To illustrate some of these statements, observe that for a cycle $C_s$, $s\geq 4$, we have $Box_1(C_s)=2$ and $Box_2(C_s)=Box(C_s)$ (see figure \ref{figc5}).
If, however, $G$ is the graph shown on the left of  Figure \ref{bg1}, then $Box_1(G)=\infty $ but $Box(G)=2$. Figure \ref{badgraph2} shows a graph $G$ where $Box(G)=2$ and $Box_1(G)=3$; in this case the realization of $G$ as $1$-boxes in $\mathbb{R}^3$ can be thought of as the edges of a cube.

\begin{figure}
\begin{center}
\includegraphics[scale=.3]{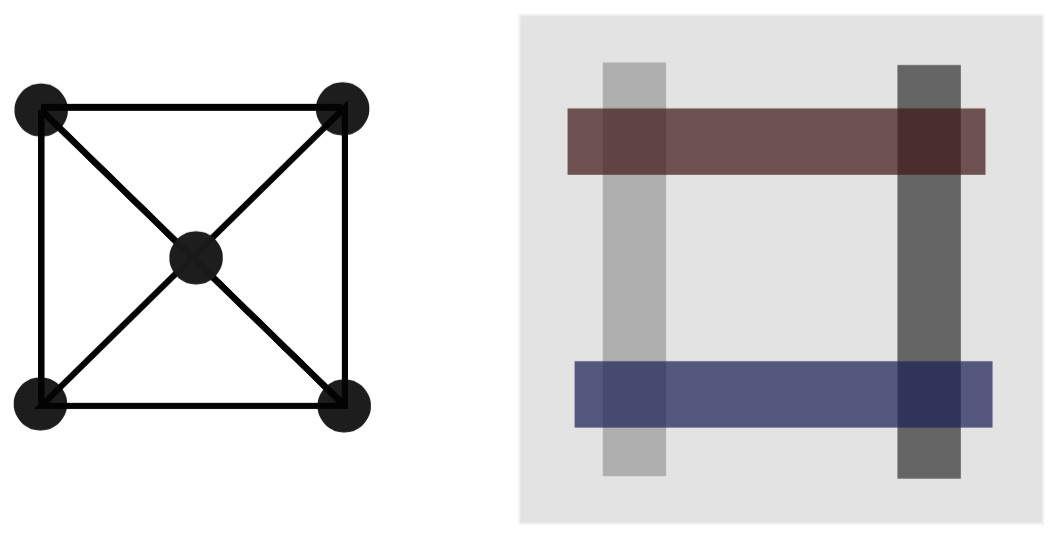}
\caption{Graph $G$ and it realization as 2-boxes in $\mathbb{R}^2$.}\label{bg1}
\end{center}
\end{figure}

\begin{figure}[h]
\begin{center}
\includegraphics[scale=.38]{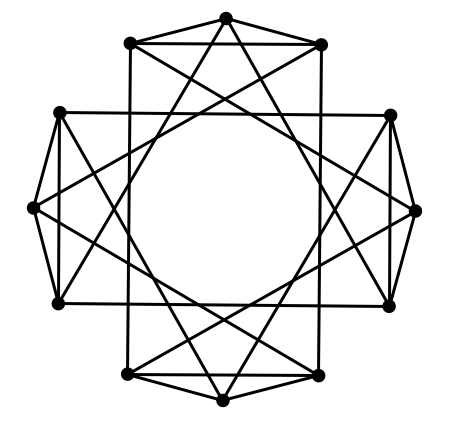}\\
\includegraphics[scale=.7]{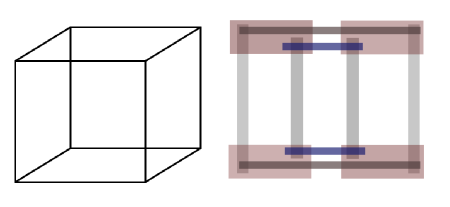}
\caption{Top: a graph $G$ where $Box(G)=2$ and $Box_1(G)=3$. Bottom left: the realization of $G$ as $1$-boxes in $\mathbb{R}^3$; each 1--box is an edge of the cube.
Bottom right: the realization of $G$ as 2-boxes in $\mathbb{R}^2$.}\label{badgraph2}
\end{center}
\end{figure}

For us, all graphs $G=(V(G),E(G))$ are finite and simple, i.e. with no loops or multiple edges. Recall that a subgraph $H\subset G$ is an induced subgraph of $G$ if $V(H)\subset V(G)$ and if $x_i,x_j\in V(H)$ and $(x_i,x_j)\in E(G)$ then $(x_i,x_j)\in E(H)$. The neighborhood of a vertex $v\in V(G)$ is the set of vertices adjacent to $v$, $N(v):=\{ w\in V(G)|(v,w)\in E(G)\}$. We will define the closure of the neighborhood of $v$  as the set of all possible edges in $N{_G}(v)$, $[N{_G}(v)]:=\{ (u,w)| u,w\in N(v)\}$.  When it is clear from  context, the graph induced by  the neighborhood of a vertex will be also referred as the neighborhood of a vertex. As usual $K_t$ will denote the complete graph with $t$ vertices. 

One consequence of the following definition and lemma is that we can determine when a graph $G$ that satisfies $Box(G)=d$ also satisfies $Box_p(G)=d$. 

\begin{defi}\label{slimproperty} We will say that a graph $G$ satisfies the $p$-slim box property in $\mathbb{R}^d$ if the following two conditions hold:
\begin{itemize}
\item[i)] $G$ is an intersection graph of $d$-boxes, i.e.  $G=\cap_{i\in I} F_i $ where every $F_i$ is an interval graph and $I:=\{1,2,\dots d\}$.
\item[ii)] For all $v\in V(G),$ there is a subset $J_v\subset I:=\{1,2,\dots d\}$ of $d-p$ indices  such that the neighborhood of $v$ is a complete subgraph,  $[N_{F_i}(v)]\subset  E(F_i)$, ${i\in J_v}$.
\end{itemize}
\end{defi}

\begin{lemma}\label{lemon}
A graph $G$ is realizable as $p$-boxes in $\mathbb{R}^d$ $(p\leq d)$ if and only if $G$ has the $p$-slim property in $\mathbb{R}^d$.
\end{lemma}
\begin{proof}

Suppose that G is realizable as $p$-boxes in $\mathbb{R}^d$. 
Let $F_i$ be the intersection graph of the projection of the realization of $G$ on the $i$th coordinate axis, and observe  that $G = F_1\cap F_2\cap\cdots\cap F_d$. Let $V$ be a $p$-box in the realization of $G$ and let  $v\in V(G)$ be its representation in the graph. Since $V$ is a $p$-box there are $d-p$ coordinate axes where the projection of $V$ is a point; hence, in any of these $d-p$ axes the corresponding $F_j$ satisfies $[N_{F_j}(v)]\subset  E(F_j)$. Therefore $G$ has the $p$-slim property in $\mathbb{R}^d$.

Suppose now that $G$ has the $p$-slim property in $\mathbb{R}^d$. Then there exist interval graphs $F_1,...,F_k$  such that $G=\cap_{i=1}^d F_i$, and for any $v\in V(G)$ there exists a subset $J_v\subset I$ of $d-p$ indices  such that  $[N_{F_i}(v)]\subset  E(F_i)$, ${i\in J_v}$. Let $I_j$ be the realization of $F_j$ as intervals. Observe that for any ${i\in J_v}$ we can reduce $v$ in $I_i$ to a point since $[N_{F_i}(v)]\subset  E(F_i)$. We observe that the intersection of the projection of all $I_j$ as the $j$-th axis is a family of $p$--boxes in $\mathbb{R}^d$ with intersection graph $G$. \end{proof}

\section{Flat boxes and Piercing numbers}

For the remainder of this paper $C_s$ will denote the cycle of length $s$, with 
$V(C_s)=\{v_1,v_2,...,v_{s}\}$ and  $E(C_s)=\{(v_1,v_2), (v_2,v_3), ..., (v_{s-1},v_{s}),(v_{s},v_1)\}$. We will denote by $P_{\{v_1,v_2, \dots , v_l \} }$ the path of length $l-1$ with vertices $V(P_{\{v_1,v_2,\dots ,v_l \} })=\{v_{1},v_{2},\dots,v_{l} \}$ and edges $E(P_{\{v_1,v_2,\dots ,v_l \}} )=\{ (v_{1},v_{2}),\dots,(v_{l-1},v_{l})\}$. We denote the complement of $C_s$ by $C^c_s$. One graph  of particular interest to us is the path $P_{v_k}(C_s):= P_{\{v_{(k+2\mod s)},v_{(k+3\mod s)},\dots,v_{(k+s-2\mod s)}\}}$ as a subgraph of $C_s$ with respect to some vertex $v_k$ of $C_s$ (see Figure \ref{Pvks}). \\

We will say that a graph $G$ is {\it m--forbidden} in  $\mathbb{R}^d$ if $G^c$ is not realizable as  $m$-boxes in $\mathbb{R}^d$. 

\begin{figure}[h]
\begin{center}
\includegraphics[scale=.35]{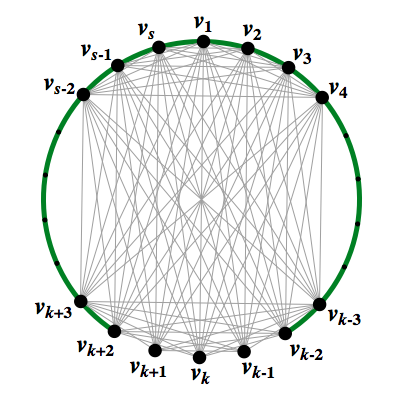}
\caption{The path $P_{v_k}(C_s)$ colored in green.}\label{Pvks}
\end{center}
\end{figure}
The following observation is well known.
\begin{obs}\label{conditions} 
Since interval graphs do not contain $C_4$ as an induced subgraph, then $2K_2$ (two disjoint edges) and a path of length 
greater or equal to $4$, are $1$--forbidden in $\mathbb{R}^1$.
\end{obs}

We will prove some lemmas that will help us to prove Theorem \ref{flatboxes}. The following observation will be really helpful to prove such lemmas.

\begin{obs}\label{path}  Suppose that $G=C^c_s$ satisfies the $p$-slim box property in $\Rd$ for some $p\leq d$ and some $s$. By  Definition \ref{slimproperty} we know that 
$G=\cap_{i\in I} F_i $ where $F_i$ is an interval graph, $I=\{1,...,d\}$,  and for every $v_k\in V(G)$ there exists $J_{v_k}\subset I$, with $|J_{v_k}|=d-p$,  such that $[N_{v_k}]\subset F_j$ for all $j\in J_{v_k}$. Then $P_{v_k}(C_s)\subset \cap_{j\in J_{v_k}} F_j$. 
\end{obs}

The following lemmas will allow us to prove the main theorem.

\begin{lemma}\label{C7}
 The cycle $C_{s}$ with $s\geq 7$ is 1-forbidden in $\mathbb{R}^d$ for any $d>1$.
\end{lemma}
\begin{proof}
We need to show that $G:=C^c_{s}$, $s\geq 7$,  is not realizable as $1$-boxes in $\Rd$. Suppose that $G$ is realizable as 1-boxes in 
$\mathbb{R}^d$. Then by Lemma \ref{lemon}, $G$ satisfies the $1$-slim box property. 
Thus $G=\cap_{i\in I}F_i$, $I=\{1,...,d\}$, where $F_i$ are interval graphs and for any vertex  $v\in V(G)$ there is a subset $J_v\subset I$ with $|J_v|=d-1$,  such that  $[N_{F_i}(v)]\subset  E(F_i)$, ${i\in J_v}$.

Without loss of generality assume that 
$J_{v_1}:=\{1,2\dots d-1\}$. 
 By Observation \ref{path}, $P_{v_1}(C_s)\subset \cap_{j=1}^{d-1}F_j$ and thus 
$(v_3,v_4),(v_4,v_5),..,(v_{s-2},v_{s-1}) \notin E(F_d)$ otherwise $G\not=\cap_{i\in I}F_i$ (see Figure \ref{projections}). 

If $s>7$, then there would be  at least four edges missing in $F_d$. This contradicts Observation \ref{conditions}. 
If $s=7$, by Observation \ref{path} on $v_{6}$ the path  
$P_{v_6}(C_s)$  should belong to $d-1$ of the $F_i$ interval graphs, but since $(v_3,v_4)\notin E(F_d)$, $J_{v_1}=J_{v_6}=\{1,...,d-1\}$.
Then $P_{v_6}(C_s)\in \cap_{i=1}^{d-1}F_i$. This implies $(v_2,v_3)\notin E(F_d)$, otherwise $(v_2,v_3)$ belongs to every $F_i$ contradicting $G=C_7^c$. Then the edges $(v_2,v_3),(v_3,v_4),(v_4,v_5),(v_5,v_6)$ are not in $E(F_d)$, contradicting Observation \ref{conditions}.  
\end{proof}

\begin{figure}[h]
\begin{center}
\includegraphics[scale=.3]{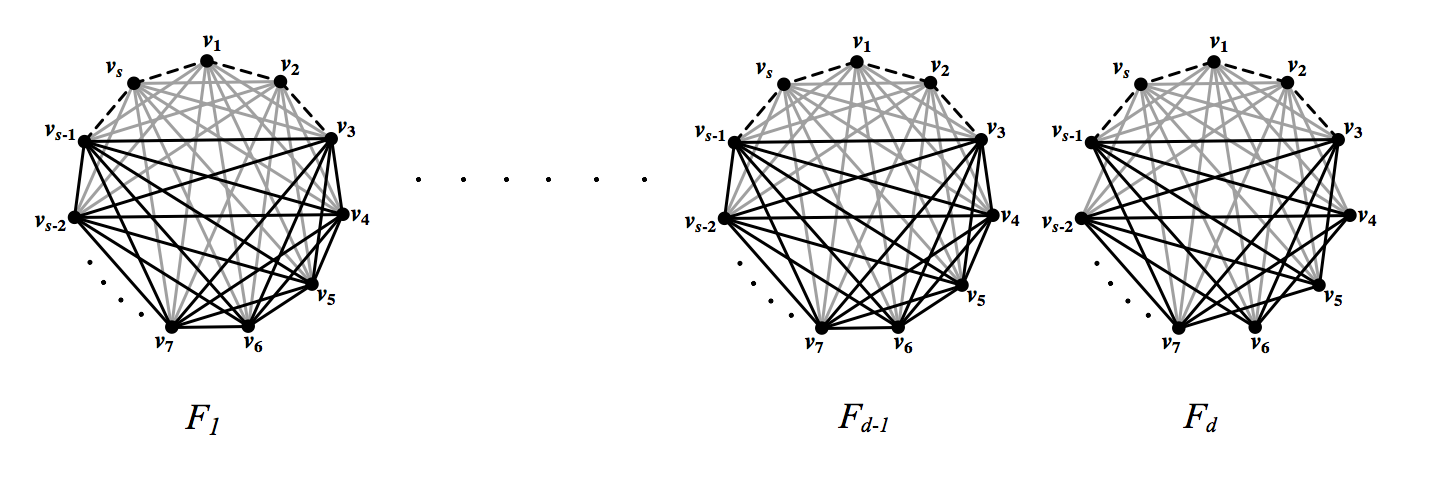}
\caption{The $d$ interval graphs whose intersection is $G$. Until this point of the proof we do not know anything about the path $P_{s-1,s,1,2,3}$ in each of the graphs, so a dashed edge means that such an edge may or may not be in the projection.
 Observe that the first $d-1$ graphs contain $[N(v_1)]$ and therefore $P_{v_1}(C_s)$, but the last graph does not contain  $P_{v_1}(C_s)$, otherwise the intersection would not be $G$.}\label{projections}
\end{center}
\end{figure}

\begin{lemma}\label{C9}
 The cycle $C_{s}$, $s\geq 9$, is 2-forbidden in $\mathbb{R}^d$ for any $d>2$.
 \end{lemma}
 \begin{proof}
Again we need to show that  $G:	=C^c_{s}$ 
is not realizable as $2$-boxes in $\Rd$.

Suppose that $G$ is realizable as 2-boxes in $\mathbb{R}^d$. Then by Lemma \ref{lemon}, $G$ satisfies the $2$-slim box property in $\mathbb{R}^d$.  
Thus $G=\cap_{i\in I}F_i$, $I=\{1,...,d\}$, where $F_i$ are interval graphs and for any vertex  $v\in V(G)$ there is a subset $J_v\subset I$, $|J_v|=d-2$,  such that  $[N_{F_i}(v)]\subset  E(F_i)$, ${i\in J_v}$. Again without loss of generality suppose   $J_{v_1}:=\{1,2,\dots, d-2\}$. By Observation \ref{path}, $P_{v_1}(C_s)\subset \cap_{j=1}^{d-2}F_j$, and thus $(v_3,v_4),(v_4,v_5),...,(v_{s-2},v_{s-1}) \notin E(F_d)\cap E(F_{d-1})$.

Without loss of generality suppose $(v_{s-2},v_{s-1})\notin E(F_{d-1})$.  Since $s\geq 9$, by Observation \ref{conditions} $(v_3,v_4),(v_4,v_5)\notin E(F_{d})$ and therefore $(v_{s-2},v_{s-3})\notin E(F_{d-1})$. 

Note that if $(v_{s-3},v_{s-4})\notin E(F_{d-1})$, $J_{v_{s-1}}=\{1,2,\dots,d-2\}$. This is because $(v_3,v_4)\notin E(F_{d})$ and $(v_{s-3},v_{s-4}),(v_{3},v_{4})\in [N_{v_{s-1}}]$; by the $p$--slim property  there exists $J_{v_{s-1}}\subset I$  where $[N_{v_{s-1}}]\subset F_j$ for all $j\in J_{v_{s-1}}.$

 Therefore by Observation \ref{path} any edge of the path $P_{v_{s-1}}(C_s)$ is not in  $E(F_d)\cap E(F_{d-1})$ (in particular this includes the edge $(v_1,v_2)
$). This is a contradiction, since by Observation \ref{conditions} the edge $(v_1,v_2)$  must be in  $E(F_d)\cap E(F_{d-1})$ otherwise we obtain either an empty path of size at least 4 or a disjoint path in any of $E(F_d), E(F_{d-1})$.
 We proceed with an analogous argument if $(v_{s-3},v_{s-4})\notin E(F_{d})$.  \end{proof}

Let $p,d\in\mathbb{N}$ with $2<p<d$. 
Suppose that $G:=C^c_{s}$ is realizable as $p$-boxes in $\mathbb{R}^d$ for some $s\geq 3p+1$. Then by Lemma \ref{lemon}, $G$ satisfies the $p$-slim box property in $\mathbb{R}^d$. 
Thus there exists $d$ interval graphs  $F_1,F_2,...,F_d$ such that $G=\cap_{i\in I}F_i$, $I=\{1,...,d\}$  and for any vertex  $v\in V(G),$ there is a subset $J_v\subset I$, $|J_v|=d-p$,  such that  $[N_{F_i}(v)]\subset  E(F_i)$, ${i\in J_v}$. Without loss of generality suppose that $J_{v_1}:=\{1,2\dots d-p\}$. Denote the rest of the vertices by $J=\{d-p+1,...,d\}$ and observe that $|J|=p$.  Since $d>p$, $J_{v_1}$ is not empty. 

 By Observation \ref{path} the path of length $s-4$  is contained in $F_i$ with $i\in J_{v_1}$, i.e.  
 $P_{v_1}(C_s)\subset \cap_{j=1}^{d-p}F_j$.  Since $G=\cap_{i\in I}F_i$, then for every edge $e$ of the path  $P_{v_1}(C_s)$ there exists at least one $i\in J$
such that $e\notin E(F_i)$, in this case we say that $e$ is a \emph{missing edge} of $F_i$. Similarly a \emph{missing path}, is a path of missing edges in $F_i$.   

We also say that $F_j$, with $j\in J$, satisfies the \emph{missing  property} if there is a missing edge of $F_j$, $e\in E(P_{v_1}(C_s))=E(P_{v_3,\dots ,v_{s-1}})$,  such that $e\in E(F_i)$ for all $i\in J\setminus\{j \}$. For example, in Figure \ref{parti} $F_3$ has the missing property since $e:=(v_9,v_{10})\notin E(F_3)$ but $e\in E(F_4)$ and $e\in E(F_5)$.

If there are  $r,t\in J$ such that either  the path $P_{v_3,v_4,v_5,v_6}$ is a missing path of $F_r$ and the path $P_{v_{s-3},v_{s-2},v_{s-1}}$ is a missing path of $F_t$ or the path $P_{v_3,v_4,v_5}$ is a missing path of $F_r$ and the path $P_{v_{s-4},v_{s-3},v_{s-2},v_{s-1}}$ is a missing path of $F_t$, we say that $J$ satisfies the \textit{extreme condition}.  
If there exists  $u$ such that  $6\leq u\leq s-4$ and  $r,t\in J$ such that $P_{v_{u-3},v_{u-2},...,v_{u+2},v_{u+3}}$ is missing in $F_t\cup F_r$, we say that $J$ satisfies the \textit{contiguous condition}. 
The following two technical lemmas imply that if $J$ satisfies the missing property for all $j \in J$, then neither the contiguous nor the extreme condition holds. 

\begin{lemma}\label{Claim 1}Suppose that for all  $j\in J$, $F_j$ satisfies the missing property (where $J$ and $F_j$ are defined as above). Then $J$ does not satisfy the contiguous condition.  
\end{lemma}

 \begin{proof}
Suppose the opposite. 
Since the edges $(v_{u},v_{u-2})$, $(v_{u},v_{u-3})$ and  $(v_{u-2},v_{u-3})$ are missing in, say, $F_r$, then $r\notin J_{v_{u}}$. 
Analogously   $(v_{u},v_{u+2})$,$(v_{u},v_{u+3})\in E(G)$ and  $(v_{u+2},v_{u+3})$ is missing in $E(F_t)$, so $t\notin J_{v_{u}}$. 
For any other $j\in J\setminus\{r,t\}$,  the missing property implies there exists a missing edge $(x,y)$ of $F_j$ such that $e\in F_{r}$ and $e\in F_{t}$. This implies $x,y\notin \{v_{u-2},v_{u-1},v_{u},v_{u+2},v_{u+2}\}$. Hence $(v_{u},x),(v_{u},y)\in E(G)$ and thus $j\notin J_{v_{u}}$. Therefore $J_{v_{u}}=J_{v_1}=\{1,...,d-p\}$. From  $6\leq u\leq s-4$ and  Observation \ref{path}, we note that  $P_{v_{u}}$ is in $F_k$ for all $k\in J_{v_1}$. That is, $F_k$ is a complete graph for all $k\in J_{v_1}$. Thus for any edge $e$ of the  the cycle $C_s$, $e\notin\cap_{j\in J}F_j$.  

Given $s\geq 3p+1$ and  $|J|=p$, there is some  $F_j$ with $j\in J$  missing more than 3 edges, which contradicts Observation \ref{conditions}.  \end{proof}

\begin{lemma}\label{Claim 2}Suppose that for  all $j\in J$, $F_j$ satisfies the missing property.   Then $J$ does not satisfies the extreme condition.
\end{lemma}

 \begin{proof}
Suppose that there are $r,t\in J$ such  that $P_{v_3,v_4,v_5,v_6}$ is a missing path of $F_r$ and  $P_{v_{s-3},v_{s-2},v_{s-1}}$ is a missing path of $F_t$.
 
 Since  $(v_3,v_5),(v_3,v_6)\in E(G)$ but $(v_6,v_5)$ is a missing edge of $F_r$, we have $r\notin J_{v_3}$. For any other
$j\in J$, by the missing property there exists a missing edge $(x, y)$ of $F_j$ such that
$e \in F_r$. In particular we observe that $(v_3, x),(v_3, y) \in E(G)$, and thus
$j \notin J_{v_3}$. Therefore $J_{v_3} = J_{v_1} =\{1,2,...,d-p\}$ and $(v_{s-1},v_s)$ is a missing edge of $\cap_{j\in J} F_j$. 
 
 By Observation \ref{conditions} and the missing property, the edge $(v_{s-1},v_s)$ is only a missing edge of $F_t$. Therefore  $P_{v_{s-3},v_{s-2},v_{s-1},v_s}$ is a missing path of $F_t$. Since  $(v_{s},v_{s-3}),(v_{s},v_{s-4})\in E(G)$, but $(v_{s-3},v_{s-4})$ is a missing edge of $F_t$, it follows $t\notin J_{v_s}$. 
 For any other $j\in J$, by the missing property there exists a missing edge $(x,y)$ of $F_j$ such that $e\in F_{t}$. In particular $(v_{s},x),(v_{s},y)\in E(G)$, so $j\notin J_{v_s}$. Therefore $J_{v_{s}}=J_{v_1}=\{1,...,d-p\}$. By Observation  \ref{path}  $P_{v_s}(C_s)$  is in $F_k$ for all $k\in J_{v_1}$, which implies that the edge $(v_2,v_3)$ is a missing edge of $\cap_{j\in J}F_j$. By Observation \ref{conditions} and the fact that for any $j\in J$, $F_j$ satisfies the missing property, there  cannot be some $F_j$, $j\in J$, with such a missing edge. This is a contradiction.\\
  Note that an analogous argument can be used if  the path $P_{v_3,v_4,v_5}$ is a missing path of $F_r$ and the path $P_{v_{s-4},v_{s-3},v_{s-2},v_{s-1}}$ is a missing path of $F_t$, for some $r,t\in J$. 
  \end{proof}

We apply Lemmas \ref{Claim 1} and \ref{Claim 2} to show the following useful lemma. 

\begin{lemma}\label{Cs}
 The cycle $C_{s}$  is $p$-forbidden in $\mathbb{R}^d$ for any $d>p>2$ when the following two conditions hold
\begin{enumerate}

 \item[i)] $s\geq 3p+1$ for $p$ even, 
 \item[ii)] $s\geq 3p+2$  for  $p$ odd.
 
 \end{enumerate} 
 \end{lemma}
 \begin{proof}

We will show that $G:=C^c_{s}$  for the corresponding values of $s$ 
is not realizable as $p$-boxes in $\mathbb{R}^d$.
Suppose that $G$ is realizable as $p$-boxes in $\mathbb{R}^d$. By Lemma \ref{lemon}, $G$ has the $p$--slim property. Then there exists $d$ interval graphs  $F_1,F_2,...,F_d$ such that $G=\cap_{i\in I}F_i$, $I=\{1,...,d\}$  and, as before, we assume  that for $v_1$, $J_{v_1}:=\{1,2\dots d-p\}$ and  $J=\{d-p+1,..,d\}$ are the rest of indices (observe that $|J|=p$).  Since $d>p$, $J_{v_1}$ is not empty. So $P_{v_1}(C_s)$ has length $s-4$, and satisfies $P_{v_1}(C_s)\subset \cap_{i\in J_{v_1}}F_i$ and each edge of $P_{v_1}(C_s)$ is missing in $ \cap_{j\in J}F_j$. 
 
Observe first that  if $s-4\geq  3p+1$, at least one of the $F_j$, $j\in J$, will contain a path with $4$ or more missing edges yielding a contradiction of Observation \ref{conditions}. Thus we may assume 
  $$3p+1\leq s< 3p+5 {\text{ for any }} d>p>2.$$
  
Furthermore, by Observation \ref{conditions} each $F_i$, $i\in J$, can only have a missing  subpath of  $P_{v_1}$ of length 3, 2 or  1. 
  
Since  $s< 3p+5$ and for any edge $e\in E(P_{v_1})$ there exists  $j\in J$ such that $e$ is missing in $F_j$, it is possible to construct a partition (not necessarily unique) of $P_{v_1}(C_s)$ into disjoint connected subpaths of length 1, 2 and 3, with the following two properties: \\

\begin{enumerate}
\item[1.] If $Q$ is an element of such a partition then $Q$ is a missing path of $F_j$ for some $j\in J$. In this case we may say that  $F_j$ represents $Q$.

\item[2.] If $Q$ and $P$ are different elements of the partition then $i\not=j$ for its corresponding representative  $F_i$ and $F_j$.\\
\end{enumerate}
For example, in Figure \ref{parti} we see a particular case when $s=11,$ $p=2$ and $d=5$. Here $J_{v_1}=\{1,2\}$ and $J=\{3,4,5\}$. The dashed edges may or may not be on the graph. Observe that  $P_{v_1}(C_{11})=P_{v_3,...,v_{10}}$ is in $F_1,F_2$. Consider the following partition $K$ of $P_{v_1}(C_{11})$: $Q_3:=P_{v_7,v_8,v_9,v_{10}}$ represented by $F_3$, $Q_4:=P_{v_6,v_7}$ represented by $F_4$ and $Q_5:=P_{v_3,v_4,v_5,v_{6}}$ represented by $F_5$. Clearly we may have also chosen the partition $K'$ as $Q_3:=P_{v_8,v_9,v_{10}}$ represented by $F_3$, $Q_4:=P_{v_6,v_7,v_8}$ represented by $F_4$ and $Q_5:=P_{v_3,v_4,v_5,v_{6}}$ represented by $F_5$.

For a given partition, let $J_i\subset J$ for $i=1,2,3,$ be the set of indices $j\in J_i$ such that $F_j$ represents an element of the partition of size $i$. 
 We observe that for any partition, $\cap_{i=1}^3 J_{i}=\emptyset$ and $\cup_{i=1}^3 J_i\subseteq J$. Furthermore,
\begin{equation}\label{system}
\begin{array}{lcl} s-4 & = & |J_1|+2|J_2|+3|J_3| \\ p&\geq &|J_1|+|J_2|+|J_3| \end{array}
\end{equation}

\begin{figure}
\includegraphics[scale=.24]{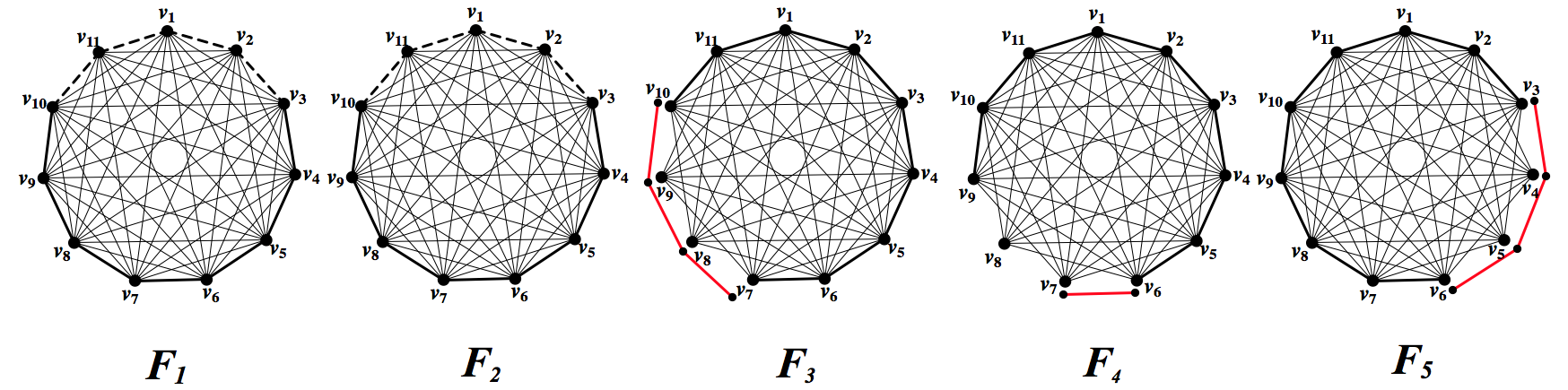}
\caption{In this figure we see a particular example of the partition $K$ when $s=11,p=2,d=5$ and, $J=\{3,4,5\}$. The red paths represent the partition $K$ formed by $Q_3$,$Q_4$ and $Q_5$. }\label{parti}
\end{figure}
In the example in Figure \ref{parti}, for the partition $K$, we obtain $J_1=\{4\}$, $J_2=\emptyset$ and $J_3=\{3,5\}$. For the partition $K'$ we obtain  $J_1=\emptyset$, $J_2=\{3,4\}$ and $J_3=\{5\}$.\\
\smallskip

In general, observe that if for all possible partitions  we have 
\begin{equation}\label{equation} |J_1|+|J_2|+|J_3| =p,
\end{equation} 
then the missing property holds for all $j\in J$. To establish this, assume that there is $j\in J$ such that for any edge $e$ missing in $F_j$, there exists $i_e\in J$ such that $e$ is missing in $F_{i_e}$. We can construct a partition with  $j\notin J_1\cup J_2\cup J_3$, and  thus this partition of $P_{v_1}(C_s)$  satisfies $|J_1|+|J_2|+|J_3| =p-1$ and $s-4  = |J_1|+2|J_2|+3|J_3| $. This  is a contradiction. 

Next we will show that in most of the cases for  $p$ and $s$, Equation (\ref{equation}) holds and therefore the missing property holds for all $j\in J$. Moreover, with appropriate values of $s$, this yields a contradiction of Lemmas \ref{Claim 1} and \ref{Claim 2}.\\

Observe that if $s=3p+4$,  the only solution for System (\ref{system}) is $|J_3|=p,|J_2|=0$ and $|J_1|=0$. This implies the missing property holds for all $j\in J$. Note also that since all elements  in the partition  have length  $3$, and since each element of the partition is uniquely represented by some $F_j$ with $j\in J$, $J$ satisfies the contiguous condition and thus contradicts Lemma \ref{Claim 1}. Therefore we may assume that $s<3p+4$. \\ 

For $s=3p+3$, there are at least $s-4=3p-1$ missing edges in $\cap_{j\in J}F_j$ (one per edge in $P_{v_1}$).  By solving System (\ref{system}) we find that $|J_3|=p-1$,$|J_2|=1$ and $|J_1|=0$, which implies that the missing property holds.

If $p>3$, then $p-1>2$. By the pigeon hole principle and the fact that each element of the partition is uniquely represented by some $F_j$ with $j\in J$, $J$ satisfies the contiguous condition and contradicts Lemma \ref{Claim 1}. If $p=3$, $|J|=3$, $|J_3|=2$ and $|J_2|=1$. Thus $J$ satisfies either the contiguous condition or the extreme condition, contradicting either Lemma \ref{Claim 1} or Lemma \ref{Claim 2} (again, this is also using the fact that each element of the partition is uniquely represented by an $F_j$ with $j\in J$).  Therefore we may assume $s\leq 3p+2$.\\

For $s=3p+2$, there are at least $s-4=3p-2$ missing edges in $\cap_{j\in J}F_j$. Solving System (\ref{system}) we find that either $|J_3|=p-1$, $|J_2|=0$ and $|J_1|=1$ or $|J_3|=p-2$, $|J_2|=2$ and $|J_1|=0$, which implies that the missing property holds for all $j\in J$.
If $p>5$, then $p-2>3$. As before, by the pigeon hole principle $J$ satisfies the contiguous condition  and contradicts Lemma \ref{Claim 1}. Note that if $p=5$, either $|J_3|=4$ and $|J_1|=1$, or $|J_3|=3$ and $|J_2|=2$. In either case, the pigeon hole principle implies $J$ satisfies either the contiguous condition or the extreme condition and thereby contradicts either Lemma \ref{Claim 1} or Lemma \ref{Claim 2}.\\

If $p=4$,  solving System \eqref{system} gives $|J_3|=3$, $|J_2|=0$, and $|J_1|=1$ or $|J_3|=2$, $|J_2|=2$, and $|J_1|=0$. In the first case, by the pigeon hole principle $J$ satisfies the contiguous condition and contradicts Lemma \ref{Claim 1}. In the second case, again by the pigeon hole principle  $J$ satisfies the extreme condition and contradicts Lemma \ref{Claim 2}.\\

If $p=3$, solving System \eqref{system} gives either $|J_3|=2$, $|J_2|=0$, and $|J_1|=1$ or $|J_3|=1$, $|J_2|=2$, and $|J_1|=0$. In the first case,  by the pigeon hole principle $J$ satisfies the contiguous condition and contradicts Lemma \ref{Claim 1}.  For the second case we observe that the only way $J$ does not satisfy the extreme condition (which would contradict Lemma \ref{Claim 2}) is if there is an $F_j$   missing no more than  $P_{v_3,v_4,v_5}$ from $P_{v_1}(C_{11})$, $F_i$ missing no more than $P_{v_5,v_6,v_7,v_8}$ from $P_{v_1}(C_{12})$, and $F_k$ missing no more than $P_{v_8,v_9,v_{10}}$ from $P_{v_1}(C_{12})$ for $i,j,k\in J$. Without loss of generality suppose that $i=d-2$, $j=d-1$ and $k=d$.  Since we know that the path $P_{v_{11}}(C_{11})$ must exist in at least $d-3$ projections, we have  $J_{v_{11}}=\{1,2,...,d-3\}$. Therefore $(v_{2},v_{3})\in E(F_i)$ for all $i\in\{1,2,...,d-3\}$. This implies  $(v_2,v_3)$ is a missing edge of $F_{d-2}\cap F_{d-1}\cap F_d$. By Observation \ref{conditions} $(v_2,v_3)$ can be only missing in $F_{d-2}$. Analogously we observe that $J_{v_2}=\{1,2,...,d-3\}$ and thus $(v_{10},v_{11})$ is a missing edge of $F_{d}$. This implies $J_{v_5}=\{1,2,...,d-3\}$, since the path $P_{v_{5}}(C_{11})$ is in $d-3$ of the $F_i$'s. Therefore $F_1,...,F_{d-3}$ are complete graphs $K_{11}$. But then $(v_{1},v_{2})$ is a missing edge of $F_{d-2}\cap F_{d-1}\cap F_d$, which is a contradiction  of Observation \ref{conditions} since there cannot be more than three edges missing in any of $F_{d-2}, F_{d-1}$ and $F_{d}$. \\

 We have shown that $s\geq 3p+2$  is $p$-forbidden for odd $p$. Thus we may assume that $s\leq 3p+1$.

 Let $s=3p+1$. Then there are $3p-3$ edges of $P_{v_1}$ missing in $\cap_{j\in J}F_j$. Solving System (\ref{system}) gives 
$$\text{ $|J_3|=p-1$, $|J_2|=0$ and, $|J_1|=0$,}$$
  $$\text{ $|J_3|=p-2$, $|J_2|=1$  and, $|J_1|=1$},\text{ or }$$
  $$\text{$|J_3|=p-3$, $|J_2|=3$ and, $|J_1|=0$.}$$

Observe that the  case where $|J_3|=p-1$, $|J_2|=0$, and $|J_1|=0$ is the only case where the missing property does not hold for all $j\in J$.  Without loss of generality suppose $J_3=\{d-p+1,...,d-1\}$. Since $G=\cap_{i=1}^d F_i$, there exists $k\in\{1,...,d\}$ such that $(v_{s-1},v_s)$ is a missing edge of $F_k$.  By Observation \ref{conditions}, $k\notin \{d-p+1,...,d-1\} $. Similarly there exists $r\in\{1,...,d\}$ such that the edge $(v_2,v_3)$ is missing in $F_r$.  By Observation \ref{conditions}, $r\notin \{k,d-p+1,...,d-1\} $. Since $|J_3|=p-1$, $|J_2|=0$, and $|J_1|=0$, we have $(v_3,v_4),(v_4,v_5)$ and $(v_5,v_6)$ are missing in some $F_t$ with $t\in J_3$. Similarly $(v_6,v_7),(v_7,v_8)$ and $(v_8,v_9)$ are missing in some $F_q$ for some $q\in J_3$. Thus there is no $J_{v_6}\subset \{1,...,d\}$ such that $|J_{v_6}|=d-p$ and such that the neighborhood of $v_6$ is a complete subgraph, given that we know $|\{r,k,d-p+1,...,d-1\}|=p+1$ and in any of $F_i$, $i\in \{r,k,d-p+1,...,d-1\}$, the neighborhood of $v_6$ is not a complete subgraph.  This yields a contradiction.

 Now suppose that $|J_3|=p-2$, $|J_2|=1$  and $|J_1|=1$. Note that the missing property holds for every $j\in J$. Let  $p>3$.
   Then Lemma \ref{Claim 1} or Lemma \ref{Claim 2} are contradicted  since $J$ would satisfy either the contiguous or the extreme condition unless there are $r,t\in J$ where $F_r$ is missing in at most the edge $(v_{s-1},v_{s-2})$ from $P_{v_1}(C_{s})$ and $F_t$ is  missing in at most  $P_{v_3,v_4,v_5,v_6}$  from $P_{v_1}(C_{s})$ (or similarly,  $F_r$ is  at most missing  $P_{v_{s-1},v_{s-2},v_{s-3},v_{s-4}}$ from $P_{v_1}(C_{s})$ and $F_t$ with missing only  $(v_3,v_4)$ from $P_{v_1}(C_{s})$). Without loss of generality suppose that  $F_r$ is missing at most 
 $(v_{s-1},v_{s-2})$ from $P_{v_1}(C_{s})$ and $F_t$ is  missing at most $P_{v_3,v_4,v_5,v_6}$  from $P_{v_1}(C_{s})$. 
Since $p>3$, $p-2>1$. This implies that there is a $g\in J$ where $F_g$ is missing $P_{v_6,v_7,v_8,v_9}$ or $P_{v_{s-2},v_{s-3},v_{s-4},s_{s-5}}$. If $F_g$ is missing  $P_{v_6,v_7,v_8,v_9}$, then $J_{v_6}=\{1,2,...,d-p\}$ and therefore  $(v_1,v_2)$ is missing in $\cap _{j\in\{d-p+1,...,d\}}F_j$ which  contradicts Observation \ref{conditions}. If $F_g$ is missing  $P_{v_{s-2},v_{s-3},v_{s-4},s_{s-5}}$, then $J_{v_{s-2}}=\{1,2,...,d-p\}$ and therefore  $(v_1,v_2)$ is missing in $\cap _{j\in\{d-p+1,...,d\}}F_j$ which is a contradiction of Observation \ref{conditions}. 

Now suppose that $|J_3|=p-3$, $|J_2|=3$, and $|J_1|=0$. Again the missing property holds.  For $p\geq 6$,  by the pigeon hole principle either  Lemma \ref{Claim 1} or Lemma \ref{Claim 2} is contradicted (given that $J$ would satisfy either the contiguous or the extreme condition).  For $p=4$ ($s=3p+1=13$) we observe that $|J_3|=1$, $|J_2|=3$, and $|J_1|=0$. If there exists $j\in J$ where $F_j$ is missing $P_{v_{12},v_{11},v_{10},v_{9}}$ or $P_{v_3,v_4,v_5,v_6}$,  Lemma \ref{Claim 2} is contradicted (given that $J$ would satisfy the extreme condition). 
  
   Therefore there are $r,t\in J$ with $F_r$  at most missing  $P_{v_3,v_4,v_5}$ from $P_{v_1}(C_{13})$ and $F_t$ with  at most missing $P_{v_{12},v_{11},v_{10}}$  from $P_{v_1}(C_{13})$. Therefore $J_{v_{12}}=\{1,2,...,d-4\}=J_{v_1}$. This implies that $(v_{13},v_{12})$ and  $(v_{2},v_{3})$ are in $E(F_k)$ for all $k\in\{1,2,...,d-4\}$ and thus they are missing in $\cap _{j\in\{d-3,d-2,d-1,d\}}F_j$. By Observation \ref{conditions}, $(v_{13},v_{12})$ is missing in $F_t$ and   $(v_{2},v_{3})$ is missing in $F_r$. Observe then that either there is an $F_l$ missing the 3-path $P_{v_{10},v_{9},v_{8},v_{7}}$ or $P_{v_{5},v_{6},v_{7},v_{8}}$. Without loss of generality suppose that $F_l$ is missing $P_{v_{10},v_{9},v_{8},v_{7}}$. We observe that $J_{v_{10}}=\{1,2,...,d-4\}$. Therefore  $(v_1,v_2)$ is missing in  $\cap _{j\in\{d-3,d-2,d-1,d\}}F_j$ which contradicts Observation \ref{conditions}. Therefore the cycle $C_{s}$ with  $s\geq 3p+1$ is $p$-forbidden for even $p$.\end{proof}
 
\section{Piercing two for families of flat boxes}

In this section we prove Theorem \ref{flatboxes}. To do so, we apply the following proposition that seems to be widely known but for which we did not find a precise reference.  In any case it is easy to show and the proof is omitted. 

\begin{prop}\label{2hp}
The piercing number of a family of boxes $\mathcal{F}$  is $n$  if and only if $\chi(G_{\mathcal{F}}^c)= n$, where $\chi(G_{\mathcal{F}}^c)$ denotes the chromatic number of the complement of the intersection  
graph of $\mathcal{F}$.
\end{prop}

\begin{obs}\label{2cricro}
Let $G$ be a graph such that $\chi(G)> 2$. If for any  $v\in V(G)$ we have $\chi(G\setminus\{v\})=2$, then $\chi(G)=3$ and $G$ is an odd cycle. This is easy to see: if  $\chi(G\setminus\{v\})=2$, then $G\setminus\{v\}$ is bipartite  and thus by coloring $v$ with a third color, we obtain  $\chi(G)=3$. It is well known that  odd cycles are the only $3$-critical chromatic graphs, i.e. they are the only family of graphs with chromatic number $3$ such that when any vertex is removed the chromatic number decreases.
\end{obs}

\begin{obs}\label{obsi}
By Observation \ref{2cricro} and Proposition \ref{2hp}, if there is a family of $p$--boxes $\mathcal{F}$ in $\mathbb{R}^d$ such that any subset of $\mathcal{F}$ has piercing number $2$, but $\mathcal{F}$ has a greater piercing number, then the piercing number of $\mathcal{F}$ is 3, and $G_{\mathcal{F}}^c$ is an odd cycle and vice versa. 
\end{obs}

Now we are ready to prove Theorem \ref{flatboxes}.

 \begin{proof}  (Theorem \ref{flatboxes}) We assume $V(K_m):=\{v_1,v_2,\dots ,v_m\}$. \\
For $m=5$ define $F_1$ and $F_2$ to be two graphs with the same vertices as $K_5$ and with $E(F_1):=E(K_5)\setminus \lbrace (v_3,v_4),(v_4,v_5),( v_5,v_1) \rbrace$ and $E(F_2):=E(K_5)\setminus \lbrace (v_1,v_2) ,( v_2,v_3)\rbrace$. Observe that $F_1\cap F_2=C^c_5$, and since $F_1,F_2$ are chordal they are interval graphs. We observe that  $C^c_5$ has the $1$-slim property in $\mathbb{R}^2$ and therefore is realizable as 1--boxes in $\mathbb{R}^2$ (see the realization of this family as $1$--boxes in Figure \ref{figc5}). So there exists a family of 1--boxes in $\Rd$, $d>2$, such that any 4 elements in the family have piercing $2$ but the whole family has piercing $3$. Hence $h(d,1,2)\geq 5$. 

\begin{figure}
\begin{center}
\includegraphics[scale=.55]{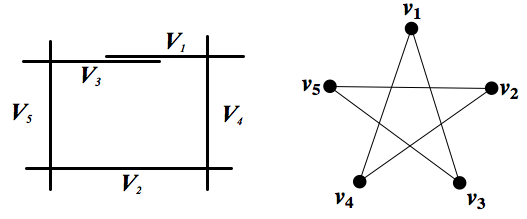}
\end{center}
\caption{Realization of $C^c_5$ as 1-boxes in $\mathbb{R}^2$}\label{figc5}
\end{figure}  

By Lemma \ref{C7}, the cycle $C_{s}$ with $s\geq 7$ is 1-forbidden in $\mathbb{R}^d$ for any $d>m$. Thus there cannot be a family of intervals in $\mathbb{R}^d$ such that its intersection graph  $C^c_{s}$ has the $1$--slim property. It follows from this and Observation \ref{obsi} that $h(d,1,2)\leq 5$.\\

For $n=7$, let $F_{1}$, $F_2$ and $F_3$ be three graphs with vertices $V(F_i)=V(K_7)$ for $i:=\{1,2,3\}$
and edges $E(F_1)=E(K_7)\setminus \lbrace(v_1,v_2),(v_2,v_3)\rbrace$, $E(F_2)=E(K_7)\setminus \lbrace(v_3,v_4),(v_4,v_5),(v_5,v_6)\rbrace$, and $E(F_3)=E(K_7)\backslash\lbrace(v_6,v_7),(v_7,v_1)\rbrace$.
Observe that $F_1\cap F_2\cap F_3=C^c_7$ and since $F_1,F_2$ and $F_3$ are chordal, they are interval graphs and thus $C^c_7$ satisfies the  $2$-slim box property in 
$\mathbb{R}^3$. Therefore $C^c_c=$ is realizable as 1--boxes in $\mathbb{R}^3$.  So there exists a family of 2--boxes in $\Rd$, $d>3$, such that any 6 elements in the family have piercing 2 but the whole family has piercing $3$ (see the realization of this family as $2$--boxes in Figure \ref{figc7}). Hence $h(d,2,2)\geq 7$. 

\begin{figure}
\begin{center}
\includegraphics[scale=.5]{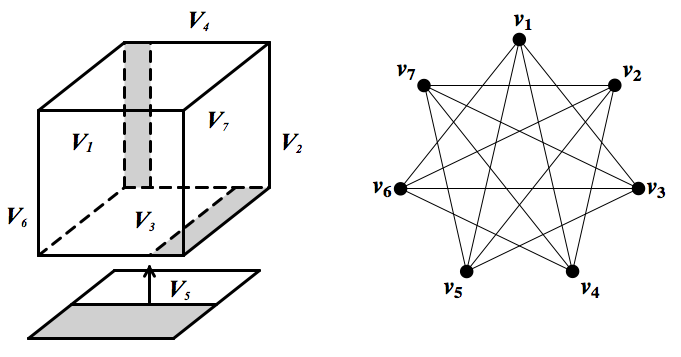}
\end{center}

\caption{Realization of $C^c_7$ as 1-boxes in $\mathbb{R}^3$, where all the facets of the cube are elements of the family except for the gray area that is missing in each of the facets. Each facet is labeled with the corresponding capital letter of the graph at the right. The facet on the $xz$-axes corresponds to $V_1$, the facet parallel to it to the right $V_2$, the top $V_4$, front $V_6$, and back $V_7$. There are two elements on the bottom ($xy$-axes), $V_3$ drawn in the cube and $V_5$ slightly below for clarification.   }\label{figc7}
\end{figure}

By Lemma \ref{C9}, the cycle $C_{s}$ with $s\geq 9$ is 2-forbidden in $\mathbb{R}^d$ for any $d>m$. Thus there cannot be a family of $2$-boxes in $\mathbb{R}^d$ such that its intersection graph is $C^c_{s}$, and along with Observation \ref{obsi} this yields $h(d,2,2)\leq 7$.\\

By Theorem \ref{DanzGrum} it is clear that $h(m,m,2)\geq 3m$ for $m$ odd and $h(m,m,2)\geq 3m-1$  for even $m$.
By Lemma \ref{Cs}, the cycle $C_{s}$  with $s\geq 3m+2$ is $m$-forbidden in $\mathbb{R}^d$ for any $m$ odd, $d>m>2$, and the cycle $C_{s}$  with $s\geq 3m+1$ is $m$-forbidden in $\mathbb{R}^d$ for any $m$ even, $d>m>2$.  Thus
 there cannot be a family of $p$--boxes in $\mathbb{R}^d$ such that its intersection graph is $C^c_{s}$. Applying Observation \ref{obsi}, the proof is complete. \end{proof}

\subsection*{Acknowledgements.}
The first author wish to acknowledge partial support  by a research assistantship, funded by
the National Institutes of Health grant R01 GM117590. 
The second author appreciate the hospitality of the department of Mathematics at U.C. Davis, support from scholarship PASPA (UNAM) and CONACYT during her sabbatical visit,  as well as acknowledge support by proyecto PAPIIT 104915, 106318 and CONACYT Ciencia B\'asica 282280.



\begin{thebibliography}{50}
\addcontentsline{toc}{chapter}{Bibliography}


\bibitem{DanzGrum} L.\ Danzer and B.\ Gr\"unbaum. \emph{Intersection
properties of boxes in $\Rd$.} Combinatorica, (2), 237--246, (1982).

\bibitem{DanzGrum&Klee} L.\ Danzer, B.\ Gr\"unbaum,  V.\ Klee. 
\emph{Helly's theorem and its relatives.} Convexity, American Mathematical
Society, (7), 101--179, (1963).

\bibitem{Dolnikov} V.\ Dol'nikov, \emph{Some Problems for Mathematical 
Olympiad Combinatorial Geometry.} Sociedad Matem\'{a}tica Mexicana--Centro de Investigaci\'on
en Matem\'aticas, A.C. (2004).

\bibitem{J.Ekoff} J.\ Eckhoff, Discrete and Computational
Geometry, The Goodman-Pollack Festschrift, \emph{Israel J. Math.} {\bf 3 },(1965), 187--198;
2) , \emph{Algorithms and Combinatorics }, Springer-Verlag,
Berlin Heidelberg New York, {\bf Vol. 25} (2003), 347--377

\bibitem{Echoff93} J.\ Eckhoff. \emph{Helly, Radon, and Caratheodory type
theorems.} Handbook of Convex Geometry, 389--448, (1993).


\bibitem{inters}
P. Erd\"os, A. Goodman, L. P\'osa, \emph{The representation of a graph by set intersections}, 
Canadian Journal of Mathematics, (18),  106-112, (1966).


\bibitem{ROB}
F.S. Roberts, \emph{On the boxicity and cubicity of a graph}, in Tutte, W. T., Recent Progress in Combinatorics, 
Academic Press,  301--310, (1969). 


\bibitem{Subram} B.V.\ Subramanya Bharadwaj, H.\ Rao Chintan, P.\
Ashok., S.\ Govindarajan. \emph{On Piercing (Pseudo)Lines and Boxes},
CCCG 2012, Charlottetown, P.E.I., 8--10, (2012).

\bibitem{Trotter}
 W. T. Trotter, {\it A characterization of Roberts' inequality for boxicity}, Discrete Mathematics, (28),  303-313, (1979).
 
\bibitem{woe}
G. Woeginger, Q. Puite, R. Pendavingh, {\it 2-piercing via graph theory}, Discrete Applied Mathematics, (156),  3510-3512, (2008). 

\bibitem{D.B. West}
D.B. West, {\it Introduction to Graph Theory}, Prentice Hall (1996).


\bibitem{sunil1}
L. Sunil Chandran, N. Sivadasan. \emph{Boxicity and treewidth.} Journal of Combinatorial Theory. Series B, (5), 733--744, (2007). 
 
\bibitem{sunil2}
 L. Sunil Chandran, M. Francis, N. Sivadasan. \emph{Boxicity and maximum degree.} Journal of Combinatorial Theory. Series B, (2), 443--445, (2008).

\bibitem{esperet}
L. Esperet, G. Joret. \emph{Boxicity of Graphs on Surfaces.} Graphs and Combinatorics, 417--427, (2013).


\bibitem{sunil3}
A. Adiga, L. Sunil Chandran, N. Sivadasan. \emph{Lower bounds for boxicity.} Combinatorica, (6), 631--655, (2014).



 

\end{thebibliography}
 \end{document}